\documentclass{article}
\usepackage{hyperref}
\usepackage{amsmath}
\usepackage{amsthm}
\usepackage{amssymb}
\usepackage{amsfonts}
\usepackage{amsxtra}
\usepackage{epsfig}
\usepackage{verbatim}
\usepackage{enumerate}
\usepackage{amscd}

\theoremstyle{plain}
\newtheorem{thm}{Theorem}
\newtheorem{prop}[thm]{Proposition}
\newtheorem{lem}[thm]{Lemma}
\newtheorem{cor}[thm]{Corollary}

\theoremstyle{definition}
\newtheorem{defn}[thm]{Definition}

\theoremstyle{remark}

\newtheorem*{rem}{Remark}

\include{header}

\newcommand{\Real}{\mathbf R}
\newcommand{\C}{\mathbf C}
\newcommand{\G}{\mathbf G}

\newcommand{\Q}{\mathbf Q}
\newcommand{\Z}{\mathbf Z}
\newcommand{\A} {\mathbf{A}}
\newcommand{\GO} {\rm{GO}}

\advance \topmargin by -\headheight
\advance \topmargin by -\headsep

\begin{document}
\title{Arithmetic properties of similitude theta lifts from orthogonal to symplectic
groups}


\author{Tobias Berger}

\date{February 26, 2008}



\maketitle
\large\begin{center}
    Preliminary Version
\end{center}

\normalsize

\begin{abstract}
By adapting the work of Kudla and Millson we obtain a lifting of cuspidal
cohomology classes for the symmetric space associated to ${\rm GO}(V)$ for an indefinite rational quadratic space $V$ of even dimension to holomorphic Siegel modular forms on ${\rm
GSp}_{n}(\A)$. For $n=2$ we prove Thom's Lemma for hyperbolic 3-space, which together with results of Kudla and Millson imply an
interpretation of the Fourier coefficients of the theta lift as period integrals of
the cohomology class over certain cycles. This allows us to prove the $p$-integrality of the lift for a particular choice of Schwartz function for almost all primes $p$.
We further calculate the Hecke eigenvalues (including for some ``bad" places) for this choice in the case of $V$ of signature $(3,1)$.
\end{abstract}

\section{Introduction}
The theta correspondence provides an important method of transfering automorphic forms between different groups. It was Shimura who started studying the arithmeticity of theta correspondences and its applications to the arithmetic of period ratios and special $L$-values. Harris and Kudla extended this work and in certain cases proved rationality of the theta lift over specified number fields (see e.g. \cite{HK92}). More recently, Prasanna proved $p$-integrality of the Jacquet-Langlands-Shimizu correspondence \cite{P06} and the Shimura and Shintani correspondences \cite{P08b}. We refer the reader to \cite{P08} for an overview of these and other results by Finis, Harris-Li-Skinner, Emerton, and Hida.
In this paper we study the $p$-integrality of the theta lift from orthogonal similitude groups of orthogonal spaces of signature $(s,t)$ with $s>t$ odd to symplectic similitude groups. The case of signature $(3,1)$ is of particular interest because of its connection to ${\rm GL}_2$ over an imaginary quadratic field (cf. \cite{HST}).

As a first step  we explain in Section \ref{stheta} how to adapt the work of Kudla and Millson
\cite{KM86}, \cite{KM87}, \cite{KM88}, \cite{KM90} on cohomological
theta lifts for isometry groups to the similitude case.
 We obtain a lifting of cuspidal
cohomology classes for the symmetric space associated to a group of orthogonal similitudes
to a holomorphic Siegel modular form on ${\rm
GSp}_n(\A)$. The results of Kudla and Millson further imply an
interpretation of its Fourier coefficients as period integrals of
the cohomology class over certain cycles. Since infinite geodesics
were not treated by Kudla and Millson we provide a proof for these
in the case of hyperbolic $3$-space in Section \ref{thomh3}. By choosing arithmetic (i.e. algebraic, rational or $p$-integral) Schwartz
functions our result allows a lifting to arithmetic Siegel modular
forms (see Section \ref{schwartz}). Here we take  as rational (or $p$-integral) structures the one coming from Betti cohomology for the orthogonal group, whilst we use the Fourier expansion for the Siegel modular form.

In Section \ref{Hecke} we calculate the Hecke action on our similitude theta lift, including for some ``bad places" dividing the level. Previously, the only calculations of the Hecke action on symplectic theta series for indefinite quadratic forms had been for signature $(2,1)$ (Waldspurger and Prasanna \cite{P08b}). One of our motivations was to obtain a $p$-integral lifting of ordinary cuspforms $\pi$ for
${\rm GL}_2$ over an imaginary quadratic field to an ordinary Siegel form, which would allow one to study $\pi$ by variational techniques using Hida families for ${\rm GSp}_2$. Combined with the Jacquet-Langlands correspondence our results should provide such a lifting for forms with unramified or Steinberg local components $\pi_v$ for $v \mid p$. We plan to investigate this and the question of non-vanishing of the theta lift in further work.

Our results carry over, mutatis mutandis, to totally real number fields and quadratic spaces of signature $((s,t),(m,0), \ldots, (m,0))$. Using the work of Funke and Millson \cite{FM06} it should be
possible to extend our results also to higher weights, i.e. the case of
cohomology classes of the similitude orthogonal group with non-trivial
coefficients and vector-valued Siegel modular forms.

The author would like to thank Chris Skinner for his suggestion to study the arithmetic of the theta lift used in the work of Harris, Soudry and Taylor \cite{HST}, and Jim Cogdell for pointing him towards the work of Kudla and Millson. Thanks also to Stephen Kudla  and Jens Funke for helpful discussions and answering questions about their work.

\section{Notation and terminology}

\subsection{Orthogonal Groups} \label{orthdefn}
Let $(V, ( \, , \,))$ be a nondegenerate quadratic space over $\Q$ of
even dimension $m$ and signature $(s,t)$ with $s > t$ and $t$ odd. Let $\chi_V(x)=(x,(-1)^{m/2} {\rm det} V):\Q^*\backslash \A^* \to \{ \pm 1\}$ be the quadratic Hecke character associated to $V$. Let
$H=\GO(V)$ denote the group of orthogonal similitudes of $V$ and let
$\lambda:H \to \G_m$ denote the multiplier character. Put $H_1={\rm
O}(V)={\rm ker}(\lambda)$. Write $H_*$ for $*=\emptyset,1$.

 We let $H_*(\Real)_+$ be the group of elements of
$H_*(\Real)$ whose image in $H_*^{\rm ad}(\Real)$ lies in its
identity component $H_*^{\rm ad}(\Real)^+$, and we let
$H_*(\Q)_+=H_*(\Q) \cap H_*(\Real)_+$. Put $K_{\infty}={\rm SO}(s) \times
{\rm SO}(t) \subset H_1(\Real)$.

Let $D$ be the symmetric space of dimension $d:=st$ given by the space of negative $t$-planes in $V(\Real)$, i.e. $$D=\{Z \in {\rm Gr}_t(V(\Real)):(\, , \, )|_Z <0 \}.$$
For an open
compact subgroup $K \subset H_*(\A_f)$ put
$$S^{*}_K=H_*(\Q)_+\backslash D \times H_*(\A_f)/K.$$ The connected components
of $S^{*}_K$ can be described as follows. Write
\begin{equation} \label{gj} H_*(\A_f)=\coprod_j H_*(\Q)_+ h_j
K\end{equation} for $h_j \in H_*(\A_f)$. Then $$S^*_K= \coprod_j
\Gamma_{j} \backslash D, $$ where $\Gamma_j$ is the image in
$H_*^{\rm ad}(\Real)^+$ of \begin{equation}\label{gammaj}
\Gamma_j'=H_*(\Q)_+ \cap h_j K h_j^{-1}.\end{equation}

Note that \begin{equation} \label{spacedecomp} \Omega^n(S^*_K) \cong
[\Omega^n(D) \otimes C^{\infty}(H_*(\A_f))]^{H_*(\Q)_+ \times K}
\cong \bigoplus_j \Omega^n(D)^{\Gamma_j},\end{equation} where the
second isomorphism is obtained by evaluation at $h_j$.

Let $\frak{g}$ be the Lie algebra of $H_1(\Real)$, and $\frak{k}$
that of $K_{\infty}$. We then have the Cartan decomposition
$\frak{g}=\frak{k}+\frak{p}$ with $\frak{p}$ the orthogonal
complement of $\frak{k}$ with respect to the Killing form.

\subsection{Symplectic groups} \label{sympldefn}
For $n \geq 1$ let $G={\rm GSp}(n) \subset {\rm GL}_{2n}$ and $G_1={\rm Sp}_n$. We also write $\lambda$
for the multiplier of $G$. For $g\in G(\A)$ put $g_1=g
\begin{pmatrix} 1 & 0
\\0& \lambda(g)^{-1}\end{pmatrix} \in G_1(\A)$. Let $G(\A)^+$ be the
subgroup of $g\in G(\A)$ such that $\lambda(g) \in \lambda(H(\A))$ and $G_+(\Real)$ the elements of $G(\Real)$ with positive multiplier.

Let $L_{\infty}\cong U(n)$  be the maximal compact subgroup of $G_1(\Real)$. Let
$\frak{g}'=\frak{k}' + \frak{p}'$ be the complexified Cartan
decomposition of $\frak{sp}_n$. Recall that $\mathbf{H}_n=\{\tau \in {\rm Sym}_n(\C): {\rm Im}(\tau) >0\} \cong G_1(\Real)/L_{\infty}$. We can identify $\frak{p}'$ with the
complex tangent space of $\mathbf{H}_n$ at the point $i {\rm 1}_n$, and the Harish-Chandra
decomposition $\frak{p}'=\frak{p}_+ \oplus \frak{p}_{-}$ gives the
splitting of $\frak{p}'$ into the holomorphic and anti-holomorphic
tangent spaces. Put $\frak{q}=\frak{k}' \oplus \frak{p}_-$.

For $L\subset G_1(\A_f)$ a compact open subgroup denote the
corresponding Shimura variety by
$$M_L=G_1(\Q) \backslash \mathbf{H}_n \times G_1(\A_f)/L.$$

\section{Theta lift} \label{stheta}

\subsection{Isometry theta lift}
Fix the nondegenerate additive character $\psi=\prod_{\ell} \psi_{\ell}$ of $\A$
trivial on $\Q$ given by $\psi_{\infty}(x)={\rm exp}(2 \pi i x)$ and, for every rational prime $\ell$, $\Psi_{\ell}(x)={\rm exp}(2 \pi i {\rm Fr}_{\ell}(x))$ for $x \in \Q_{\ell}$, where ${\rm Fr}_{\ell}(x)$ denotes the fractional part of $x$. Let
$1\leq n \leq s$. Let $\omega=\omega_{\psi}$ denote the usual
actions of $H_1(\A)$ and $G_1(\A)$ on the Schwartz-Bruhat space
$S(V(\A)^n)$ of $V(\A)^n$ (see, for example, \cite{Ro01} \S 1).

Let $K \subset H_1(\A_f)$ be a compact open subgroup. Denote by
$S(V(\A)^n)^K$ the $K$-invariant Schwartz functions. For $g \in G_1(\A), h \in H_1(\A_f)$ and $\phi \in [S(V(\A^n))^K \otimes
\Omega^{d-i}(D)]^{H_1(\Real)}$ defining a closed $d-i$-form on $D$
let
\begin{equation} \label{theta}\theta(g,h, \phi)=\sum_{x \in V(\Q)^n} \omega(g, h)  \phi(x).\end{equation}
This defines a closed $d-i$-form $\theta(g,\phi)$ on $S^1_K$.

Following \cite{KM90} we define a pairing
$$H^i_c(S^1_{K},\C) \times H^{d-i}_{\rm ct}(H_1(\Real),S(V(\A)^n)^K)\to
C^{\infty}(G(\A))$$ by \begin{equation} \label{KMeqn} \langle
[\eta],[\phi]\rangle_{K}(g)=\int_{S^1_{K}} \eta \wedge
\theta(g,\phi).\end{equation} Here we take as measure on $S^1_K$ the one induced from the measure $dh_1$ on $H_1(\Q) \backslash H_1(\A)$ defined on p. 83 of \cite{HK92}. We recall the following notation
from \cite{KM90}: If $\frak{m}$ is a Lie algebra, $\chi: \frak{m}
\to \C$ a homomorphism and $U$ and $\frak{m}$-module, then
$$U^{\frak{m}}_{\chi}=\{ u \in U: x.u=\chi(x) u \text{ for all } x
\in \frak{m}\}.$$ We denote the subspace of \textit{holomorphic}
Schwartz classes, i.e. the classes annihilated by $\frak{p}_-$, by
$$H^{d-i}_{\rm ct}(H_1(\Real),S(V(\A)^n))^{\frak{p}_-}.$$

For holomorphic automorphic forms in the adelic setting we refer the
reader to \cite{Ha84} Section 2. Let $\chi_m$ denote the character
of $L_{\infty}$ given by $\chi_m(k)={\rm det}(k)^{m/2}$. For
$L\subset G_1(\A_f)$ compact open let $L_{\infty} L$ act on
$(G_1(\Q) \backslash G_1(\A)) \times \C_{\chi_m}$ by $$(g,v)\cdot
\ell=(g \ell, \chi_m(\ell_{\infty}^{-1}) v)$$ for $g \in G_1(\Q)
\backslash G_1(\A)$ and $\ell=(\ell_{\infty}, \ell_f) \in L_{\infty}
\times L$. The projection $G_1(\Q) \backslash G_1(\A) \to M_L$
identifies the quotient $G_1(\Q) \backslash G_1(\A) \times
\C_{\chi_m}/L_{\infty} L$ with a holomorphic vector bundle
$\mathcal{L}_m$ over $M_L$. We are going to identify global sections
of $\mathcal{L}_m$ with the holomorphic automorphic forms
$\mathcal{H}_{\chi_m}(G_1(\Q) \backslash G_1(\A)/L,\C_{\chi_m})$ of
weight $m/2$ and level $L$, defined as $$\{\phi \in
C^{\infty}(G_1(\Q)\backslash G_1(\A)/L,\C_{\chi_m}):
\phi(g\ell)=\chi_m(\ell^{-1}) \phi(g), \ell \in L_{\infty}, g\in
G_1(\A); \frak{p}_- \phi=0\}.$$

Adopted to our adelic setting Theorem 1 of \cite{KM90} states:

\begin{thm} \label{KM}
The induced pairing $$\langle \, , \, \rangle_K: H^i_c(S^1_K,\C) \times
H^{d-i}_{\rm ct}(H_1(\Real),S(V(\A)^n)^K)^{\frak{q}}_{\chi_m} \to
\Gamma(\mathcal{L}_m)$$ takes values in the holomorphic sections,
i.e., in the space of Siegel modular forms of weight $m/2$.
\end{thm}

\subsection{Similitude theta lift} We adapt the definition of the similitude theta lift for automorphic
forms in \cite{HK92} to extend the cohomological theta lifts of
Kudla and Millson.

Recall from Roberts \cite{Ro01} the definition of the extended Weil
representation for
$$R(\A)=\{(g,h) \in G(\A) \times H(\A) : \lambda(g)=\lambda(h)\}.$$ For $h
\in H(\A)$ and $\varphi \in S(V(\A)^n)$, let
$$L(h) \varphi(x)=|\lambda(h)|^{-mn/4} \varphi(h^{-1} x).$$
Now $(g,h) \in R(\A)$ acts via $\omega(g,h):=\omega(g_1)L(h)$ (see Section \ref{sympldefn} for definition of $g_1$).
The actions of $G_1(\A)$ and $H(\A)$ do not commute, but as in \cite{HK92} Lemma 5.1.2 it is easy to check the following:
\begin{lem} \label{commutation}
  For $g \in G(\A)$ and $h\in H(\A)$ we have $$\omega(g_1) L(h)=L(h) \omega(\begin{pmatrix}
  1&0\\0& \lambda(g)^{-1}\end{pmatrix} g).$$
\end{lem}

Let $K \subset H(\A_f)$ be a compact open subgroup and write $K_1=K
\cap H_1(\A_f)$. Consider a $K_1$-invariant Schwartz function
$\varphi \in S(V(\A_f)^n)$. Let \begin{equation} \tilde
\varphi_{\infty} \in [S(V(\Real)^n) \otimes
\Omega^{d-i}(D)]^{H_1(\Real)}\end{equation} be a holomorphic
Schwartz form and write $\tilde \varphi=\varphi \otimes \tilde
\varphi_{\infty}$. For $h \in H(\A)$ put $L(h) \tilde \varphi:= L(h_f) \varphi \otimes
(h_{\infty}.\tilde \varphi_{\infty})$.

\begin{lem} \label{H(Q)inv}
  For $h \in H(\Q)$ with $\lambda(h) >0$, $z \in D$, and $t_z \in \Lambda^{d-i} T_z D$,
  $$\sum_{x \in V(\Q)^n} L(h) \tilde \varphi(x,h_{\infty} z, {\rm DL}_{h_{\infty}} t_z)=\sum_{x \in V(\Q)^n} \tilde \varphi(x,z,t_z).$$
\end{lem}

\begin{proof}
  Recall that the action of $h_{\infty}^{-1}$ on  $\omega \in \Omega^{d-i}(D)$ is given by $(h_{\infty}^{-1}.\omega)(z,t_z)=\omega(h_{\infty} z, {\rm DL}_{h_{\infty}} t_z)$.
  Note that we can factor $h_{\infty} \in H(\Real)$ with $\lambda(h_{\infty})>0$ as $$h_{\infty}=h_1 \cdot (\lambda(h_{\infty})^{1/2} I_m)$$ with $h_1 \in H_1(\Real)$ and that the action of $h_{\infty}^{-1}$ on $\Omega^{d-i}(D)$ is via $h_1^{-1}$ and the action of $h_{\infty}$ on $\tilde \varphi_{\infty}$ via $\lambda(h_{\infty})^{1/2} I_m \in Z(H(\Real))$. Therefore the left hand side of the Lemma equals $$\sum_{x \in V(\Q)^n} \tilde \varphi(h^{-1} x,z,t_z).$$
\end{proof}

\begin{defn} \label{thetaDefn}
Let $[\eta] \in H^i_c(S_K,\C)$ with central character $\chi_{\eta}$. For  $g\in G(\A)^+$ with $\lambda(g_{\infty})>0$
let $h\in H(\A)$ be any element such that $\lambda(h)=\lambda(g)$. Put
$$K_1^{h_f}=h_fK h_f^{-1} \cap H_1(\A_f)=h_fK_1 h_f^{-1}$$ and
$$\tilde \varphi'=L(h) \tilde \varphi \in [S(V(\A)^n) \otimes
\Omega^{d-i}(D)]^{H_1(\Real)\times K_1^{h_f}}.$$
Now define a function on $G(\A)^+ \cap G_+(\Real)$ by
\begin{equation}
\label{thetadefn}\theta_{\varphi}(\eta)(g)=\langle
[\eta(\cdot h_f)],[\tilde \varphi']\rangle_{K_1^{h_f}}(g_1).\end{equation}
\end{defn}

\begin{rem}
Note that for signature $s \neq t$ we have $\lambda(H(\Real))>0$ and so $G(\A)^+ \cap G_+(\Real)=G(\A)^+$.
\end{rem}

\begin{prop} \label{SMF}
The definition (\ref{thetadefn}) is independent of the
choice of $h$, left-invariant under $G(\Q)^+$, and
extends uniquely to a $G(\Q)$-left invariant function on $G(\A)$ with support in $G(\Q) G(\A)^+$.
This extended function, also denoted by $\theta_{\varphi}(\eta)$, is
an automorphic form on $G(\A)$ with central character $\chi_V^n
\chi_{\eta}$.
Furthermore, $\theta_{\varphi}(\eta)$ defines a holomorphic Siegel modular form
on $G(\A)$ of weight $m/2$.
\end{prop}

\begin{proof}
The independence of the definition from the choice of $h$ follows from the choice of measure on p.83 of \cite{HK92}. The statement about the central character is proven as in Lemma 5.1.9(ii) of \cite{HK92} noting that the central character of $\eta$ has finite order.

To prove that $\theta_{\varphi}(\eta)$ is left-invariant under $\gamma \in {\rm GSp}_n(\Q)^+ \cap G(\Real)_+$ let $\gamma' \in H(\Q)$ such that $\lambda(\gamma')=\lambda(\gamma)$. We essentially follow the proof of \cite{HK92} Lemma 5.1.9(i), with the modifications necessary for our cohomological definition. By definition,

$$\theta_{\varphi}(\eta)(\gamma g)= \int_{S^1_{K_1^{\gamma'_f h_f}}} \eta(z,h_1 \gamma'_f h_f) \wedge \theta((\gamma g)_1,L(h_1 \gamma' h) \tilde \varphi(z)),$$ where we write $(z,h_1)$ for a point in $S^1_{K_1^{\gamma'_f h_f}}$. By the invariance of $\eta$ under $H(\Q)$ this equals
$$\int_{S^1_{K_1^{\gamma'_f h_f}}} \eta(\gamma'^{-1}(z,h_1 \gamma'_f h_f)) \wedge \theta((\gamma g)_1,L(\gamma'_f \gamma'^{-1}_f h_1 \gamma' h) \tilde \varphi(z)).$$ Changing variables to $(z',h_1')=(\gamma'^{-1}_{\infty} z, \gamma'^{-1}_f h_1 \gamma'_f)$, which preserves the measure since it corresponds to conjugation by $\gamma'^{-1}$ on $H_1(\A)$, we obtain
$$\theta_{\varphi}(\eta)(\gamma g)=\int_{S^1_{K_1^{h_f}}} \eta((z',h'_1 h_f)) \wedge \theta((\gamma g)_1,L(\gamma' h'_1 h) \tilde \varphi(\gamma'_{\infty} z')).$$
By the standard invariance under $G_1(\Q)$ and Lemma \ref{H(Q)inv} we have $$\theta((\gamma g)_1,L(\gamma' h'_1 h) \tilde \varphi(\gamma'_{\infty} z'))=\sum_{x \in V(\Q)^n} \omega(g_1) L(\gamma' h'_1 h) \tilde \varphi(x,\gamma'_{\infty} z')=\sum_{x \in V(\Q)^n} \omega(g_1) L(h'_1 h) \tilde \varphi(x, z'),$$ which shows that
$$\theta_{\varphi}(\eta)(\gamma g)= \theta_{\varphi}(\eta)(g).$$

For the last statement of the theorem we use that $\theta_{\varphi}(\eta)$ is right-invariant by a compact
open subgroup $\tilde L \subset G(\A_f)$ and that $$ G(\A)= \coprod_j
G(\Q) g_j G(\Real)_+ \tilde L$$ for some $g_j \in G(\A_f)$. By the
transformation properties of $\theta_{\varphi}(\eta)$ it enough to
prove that $$\frak{p}_-.(\theta_{\varphi}(\eta)|_{G_1(\Real) \times
G(\A_f)}) =0.$$ But Lemma 3.3 of \cite{KM90} shows for any $g_f \in G(\A_f)$ that
$\theta_{\varphi}(\eta)|_{G_1(\Real) \times \{g_f \}}$ is a
holomorphic section of $\mathcal{L}_m$ over $M_{L}$ for some $L
\subset G_1(\A_f)$ and hence $\theta_{\varphi}(\eta)$ defines a holomorphic Siegel modular form.
\end{proof}
\section{Fourier expansion}

\subsection {Definition of weighted cycles}
We recall the following definitions from \cite{K97}. Let $U \subset V$ be a $\Q$-subspace with ${\rm dim}_{\Q} U =n$ such
that $(\, , \,)|_U$ is positive definite. Put $D_U=\{Z \in D| Z \bot U\}$
and let $H_{1,U}$ be the stabilizer of $U$ in $H_1(\Real)$. Let
$H_{1,U}^0$ be the connected component of identity of $H_{1,U}$.

For a fixed neat level $K\subset H_1(\A_f)$ and an element $h\in
H_1(\A_f)$, we define a connected cycle associated to $U$ as
follows. Set
$$\Gamma'_h=H_1(\Q)_+ \cap hKh^{-1}$$ and $$\Gamma'_{h,U}= \Gamma'_h
\cap H_{1,U}^0.$$ Let $\Gamma_h$ (resp. $\Gamma_{h,U}$) denote the
image of $\Gamma'_h$ (resp. $\Gamma'_{h,U}$) in $H_1^{\rm
ad}(\Real)^+$.

Define the map $$\Gamma_{h,U} \backslash D_U \to \Gamma_h \backslash
D: \Gamma_{h,U} z \mapsto \Gamma_h z.$$ We will denote by
$c(U,h,K)$ the connected cycle this defines in $H_{(s-n)t}(\Gamma_h \backslash D,\Z)$. Note, that if $\gamma \in H_1(\Q)_+$, then
$$\gamma \cdot c(U,h,K)=c(\gamma U,\gamma h,K),$$ where $\gamma \cdot c(U,h,K)$
denotes the cycle which is the image of the composite mapping
$$\Gamma_{h,U} \backslash D_U \to \Gamma_h \backslash U \to \gamma \Gamma_h \gamma^{-1} \backslash D.$$
Also note that $$c(U,hk,K)=c(U,h,K)$$ for all $k\in K$.

We will now define weighted sums of these connected cycles to define
cycles for $S^1_K=\coprod \Gamma_j \backslash D$. Let $\beta=\beta^t
\in M_n(\Q)$ be a positive definite symmetric matrix and, for
$x=(x_1, \ldots, x_n) \in V^n$, put $$(x,x)=((x_i,x_j)) \in {\rm
Sym}_n(\Q),$$ where $( \, , \, )$ is the symmetric bilinear form on $V$. Let
$$\Omega_{\beta}=\{ x \in V^n | \frac{1}{2}(x,x)=\beta\}$$ be the
corresponding hyperboloid.

Let $S(V(\A_f^n))_{\Z}$ be the space of locally constant $\Z$-valued
functions on $V(\A_f^n)$ of compact support. For any commutative
ring $R$, let $$S(V(\A_f)^n)_R=S(V(\A_f^n))_{\Z} \otimes_{\Z} R.$$

Motivated by \cite{K97} Proposition 5.4 we make the following
definition:

\begin{defn}
For $\varphi \in S(V(\A_f)^n)_{R}$ and $K \subset H_1(\A_f)$, a $K$-invariant
Schwartz function, let
$$Z(\beta,\varphi, K)= \sum_j \sum_{x\in \Omega_{\beta}(\Q)
\mod{\Gamma_j'}} \varphi(h_j^{-1} x) \cdot c(U(x),h_j,K) \in
H_{(s-n)t}(S^1_K,\partial S^1_K, R),$$ where $U(x)$ is
the $\Q$-subspace of $V$ spanned by the components of $x$.
\end{defn}

\subsection{Fourier coefficients}
A choice of maximal compact subgroup $K_{\infty} \subset H_1(
\Real)$ determines a base point $Z_0 \in D$ and a  positive definite
form $( \, , \, )_0$ on $V$ which is a minimal majorant of the given form
$(\, , \, )$ on $V$ of signature $(s,t)$. We define the Gaussian
$\varphi_0 \in S(V(\Real)^n)$ by
$$\varphi_0(x_1, \ldots, x_n)=\prod_{i=1}^n {\rm exp}(-\pi (x_i,x_i)_0).$$ Kudla and Millson define in \cite{KM90} a particular holomorphic
Schwartz class $[\varphi^+_{nt}] \in H^{nt}_{\rm
ct}(H_1(\Real),S(V(\Real)^n))^{\frak{q}}_{\chi_m}$ taking value in
$\mathbf{S}(V(\Real)^n)$. Here the \textit{polynomial Fock space}
$\mathbf{S}(V(\Real)^n)$ is defined to be the space of those
Schwartz functions on $V(\Real)^n$ of the form $p(v_1,, \ldots, v_n)
\varphi_0(v_1, \ldots, v_n)$, where $p(v_1, \ldots, v_n)$ is a
polynomial function on $V(\Real)^n$.

Following \cite{KM82} and \cite{FM02} (4.17) we give here the
definition of the Schwartz form $\varphi^+_{n}$ in the case of
signature $(s,1)$ (we refer the reader to \S 5 of \cite{KM90} for
the general case): For $x=(x_1, \ldots, x_n) \in V(\Real)^n$ the
Schwartz form $\varphi^+_{n}\in [S(V(\Real)^n) \otimes
\Omega^{n}(D)]^{H_1(\Real)} \cong [S(V(\Real)^n) \otimes
\bigwedge^{n} (\mathfrak{p})^*]^{K_{\infty}}$ is given by
$$\varphi^+_{n}(x,w)=2^{n/2} {\rm det}(x,w) \varphi_0(x)$$ for $w=(w_1, \ldots, w_n) \in
\frak{p}^n \cong (Z_0^{\bot})^n$, where $(x,w)$ is the matrix with
$(i,j)$-th entry $(x_i,w_j)$. (In fact, this differs from the
definition in \cite{KM90} by the factor $2^{n/2}$.) Theorem 5.2 of
\cite{KM90} proves that this gives rise to a holomorphic Schwartz
class in $H^{n}_{\rm
ct}(H_1(\Real),S(V(\Real)^n))^{\frak{q}}_{\chi_m}$.

Using the calculation of Fourier coefficients by Kudla and Millson
we are going to prove in the next sections that for $$\tilde
\varphi_{\infty}=\varphi^+_{nt}\in [S(V(\Real)^n) \otimes
\Omega^{nt}(D)]^{H_1(\Real)} \cong [S(V(\Real)^n) \otimes
\bigwedge^{nt} (\mathfrak{p})^*]^{K_{\infty}}
$$ the form $\theta_{\varphi}(\eta)$ given by Proposition \ref{SMF} is an arithmetic Siegel modular
form for arithmetic $\varphi$ and $\eta$. We fix this choice for $\tilde
\varphi_{\infty}$ from now on.

Let $(V_+, ( \, , \, )_+)$ be a positive definite inner product space of
dimension $m$ over $\Real$, and let $\varphi_+^0 \in S(V^n_+)$ be
the Gaussian $$\varphi_+^0(x)= {\rm exp}(-\pi {\rm tr}(x,x)_+).$$ If
$x \in V^n_+$ with $(1/2) (x,x)_+=\beta \in {\rm Sym}_n(\Real)$,
then for $g \in G_1(\Real)$, we define the generalized
Whittaker function \begin{equation} \label{whittaker}
W_{\beta}(g)=\omega_+(g) \varphi^0_+(x),\end{equation} where
$\omega_+$ is the Weil representation
associated to $V_+$. As in \cite{K97} (7.22) we note that for
$$g=\begin{pmatrix}1&u\\ & 1 \end{pmatrix}
\begin{pmatrix}v^{1/2}& \\ & v^{-1/2} \end{pmatrix} \ell$$  with $\ell \in L_{\infty}$
and $\tau=u+iv= g (i \cdot {\rm 1}_n) \in \mathbf{H}_n$ we have
$$W_{\beta}(g)={\rm det}(v)^{m/4} {\rm exp}(2 \pi i {\rm tr}(\beta \tau)) {\rm det}(\ell)^{m/2}.$$

\begin{thm} \label{Fourier}
Let $K \subset H(\A_f)$ be a compact open subgroup small enough such
that $-1$ is not in $K \cap H_1(\Q)_+$. Put $K_1=K \cap H_1(\A_f)$. Let $[\eta] \in
H^{(s-n)t}_c(S_K,\C)$ and $\varphi \in S(V(\A_f)^n)^{K_1}$. For $g
\in G_1(\Real) \times G(\A_f)^+$ let $h\in H(\A_f)$
such that $\lambda(h)=\lambda(g)$ and
$$\varphi'=\omega(g_f,h) \varphi \in S(V(\A_f)^n).$$
Then we have
$$\theta_{\varphi}(\eta)(g)=\sum_{\beta > 0} W_{\beta}(g_{\infty})
\cdot \int_{Z(\beta,\varphi',K^{h}_1)} \eta(h).$$
\end{thm}

\begin{rem}
  If $-1 \in K \cap H_1(\Q)_+$ then the right hand side must be multiplied by a factor of $2$.
\end{rem}

\begin{proof}
We follow the proof of Theorem 8.1 of \cite{K97}. Put
$\eta'=h.\eta$. Write $H_1(\A_f)=\coprod_j H_1(\Q)_+ h_j K^h_1$. By
(\ref{spacedecomp}) we get
$$\theta_{\varphi}(\eta)=\int_{S^1_{K^h_1}} \eta' \wedge \theta(g_{\infty},\varphi')$$
$$=\sum_j
\int_{\Gamma_j \backslash D} \eta'(h_j) \wedge \theta(g_{\infty},h_j,\varphi')
=\sum_j \sum_{\beta \in {\rm Sym}_n(\Q)} \sum_{x \in
\Omega_{\beta}(\Q) \mod{\Gamma_j'}} \int_{\Gamma_{j,x}\backslash D}
\eta'(h_j) \wedge \omega(g_{\infty}) \tilde \varphi'(h_j^{-1}x).$$
Here we use the fact, that if $\Gamma_{j,x}$ is the image in
$\Gamma_j$ of the stabilizer $\Gamma'_{j,x}$ of $x$ in $\Gamma'_j$,
then under our assumptions on $K$ we have
$$\Gamma'_{j,x} \backslash \Gamma'_j \cong \Gamma_{j,x} \backslash
\Gamma_j.$$

Now one main result of \cite{KM90}  is that the terms where $\beta$
is not positive definite vanish (this is where our assumption that
$t$ is odd comes in!). We require a form of Thom's Lemma, as stated
in \cite{KM90} Theorem 9.1, where the results of \cite{KM87} (for
$\Gamma \backslash D$ compact) and \cite{KM88} (for $\Gamma
\backslash D$ finite volume) are combined. In fact, these results do
not cover the case of an infinite geodesic, which can arise for
signature $(s,1)$. \cite{FM02} added a proof for this in the case of
signature $(2,1)$ and a principal congruence subgroup. We are going
to prove this for our main case of interest ($(s,t)=(3,1)$ and
$n=2$) in the next section.

\begin{thm}[Thom's Lemma]
Let $\beta>0$ and $x \in \Omega_{\beta}(\Q)$. Put $U=U(x)$. Let
$\Gamma_U$ be a discrete subgroup of $H_{1,U}^0$. For any closed and
bounded $(s-n)t$-form $\eta$ on $\Gamma_U\backslash D$,
$$\int_{\Gamma_U \backslash D} \eta \wedge (\omega(g_{\infty}) \tilde
\varphi_{\infty})(x) =W_{\beta}(g_{\infty}) \int_{\Gamma_U
\backslash D_U} \eta$$
\end{thm}
\end{proof}

\begin{rem}
Kudla and Millson also treat the case of even $t$. In this case
one has non-trivial Fourier coefficients for positive semi-definite
$\beta$, for which the period integral involves also powers of an
Euler form (see Theorem 9.3 of \cite{KM90}).
\end{rem}

\subsection{Thom Lemma for hyperbolic 3-space} \label{thomh3}
In this section we will prove Thom's Lemma in the special case we
are most interested in: fix $(s,t)=(3,1)$, $n=2$ throughout this
section.

Let $F=Q(\sqrt{-D})$ be an imaginary quadratic field. We denote its
ring of integers by $\mathcal{O}$. Underlying our calculation is the
accidental isomorphism $${\rm Spin}_V(\Real) \cong {\rm Res}_{F/\Q} {\rm
SL}_2(\Real).$$ On the one hand, the symmetric space $D$ in this case can be realized as $$D=\{Z \in V(\Real): (Z,Z)=-1\}^0,$$ on the other hand it is isomorphic to
hyperbolic $3$-space $\mathbf{H}_3=\C \times \Real_{>0}$, elements
of which we write as $(z,r)$ with $z=x+iy$ for $x,y \in \Real, r\in
\Real_{>0}$. The group ${\rm GL}_2(\C)$ acts on $\mathbf{H}_3$ via
hyperbolic isometries. The action is most concisely written using
quaternion notation: identifying the point $(z,r)$ with the
quaternion $q=z+rj$ the action is given by $$\begin{pmatrix}a&b\\c&d
\end{pmatrix}: q \mapsto \frac{aq+b}{cq+d}.$$

Since we only work with $V(\Real)$ in the following, we can assume without loss of generality that $V$ is the
 hermitian matrices $$V=\{ X \in M_2(F): X^t=\overline X \}$$ with quadratic form $$X \mapsto -{\rm det}(X)$$
and corresponding bilinear form $$(X,Y) \mapsto - \frac{1}{2} {\rm
tr}(X \cdot Y^*),$$ where $$\begin{pmatrix}a&b \\c&d
\end{pmatrix}^*=\begin{pmatrix}d&-b \\-c&a \end{pmatrix}.$$
The group ${\rm SL}_2(F)$ acts isometrically on $V$ by
$$X \mapsto g X \overline g^t$$ for $X\in V$ and $g\in {\rm
SL}_2(F)$.

We fix an orthonormal basis of $V(\Real)$ given by
$e_1=\begin{pmatrix}1& \\&-1\end{pmatrix}$, $e_2=\begin{pmatrix}&1
\\1&\end{pmatrix}$, $e_3=\begin{pmatrix}&i \\-i&\end{pmatrix}$, and $e_4=\begin{pmatrix}1&
\\&1\end{pmatrix}=Z_0$, with respect to which the pairing is of the form
$$(\, , \,) \sim \begin{pmatrix}1&&&\\&1&&\\&&1&\\&&&-1\end{pmatrix}.$$
We identify $\frak{p}$ with $\Real^3$ via the basis $\{e_1, e_2,
e_3\}$ for $Z_0^{\bot}$. Let $\{\omega_1, \omega_2, \omega_3\}$ be
the corresponding dual basis of $\frak{p}^*$. Choosing the basis $\{e_1, \ldots,
e_4\}$ fixes an isomorphism $V(\Real)^2 \cong M_{4,2}(\Real)$. By \cite{FM02} (4.14) we then obtain for $X=((x_{1,1}, \ldots x_{1,4}), (x_{2,1}, \ldots x_{2,4}))\in V(\Real)^2$ that
$$\varphi_2^+(X)=2 (\omega(1,X) \wedge \omega(2,X)) \cdot
\varphi_0(X)$$ with $\omega(s,X)= \sum_{i=1}^3 x_{i,s} \omega_i$.

We also pick two isotropic vectors $u_0=\begin{pmatrix}1&
\\&\end{pmatrix}$ and $u_0'=\begin{pmatrix}& \\&1\end{pmatrix}$ and
note that with respect to the basis $\{u_0, e_2, e_3, u_0'\}$ of
$V(\Real)$ the majorant associated to the base point $Z_0$ is given
by $$(\, , \,)_0 \sim
\begin{pmatrix}1/2&&&\\&1&&\\&&1&\\&&&1/2\end{pmatrix}.$$

For an ideal $\frak{n} \subset \mathcal{O}$ put $$\Gamma_0(\frak{n})=\left\{\begin{pmatrix}a&b\\c&d
\end{pmatrix} \in {\rm SL}_2(\mathcal{O}): c\in \frak{n} \right\}.$$
Let $[\eta] \in
H^1_c(\Gamma_0(\frak{n}) \backslash \mathbf{H}_3, \C)$  be the class
of a rapidly decreasing harmonic form $\eta$.
Writing $\varphi(X,(z,r))$ for the $2$-form on $D$ corresponding to
$\varphi_2^+$ we need to show
$$\int_{\Gamma_U \backslash D} \eta(z,r) \wedge \varphi(X, (z,r))= {\rm
exp}(-\pi {\rm tr}(X,X)) \int_{\Gamma_U \backslash D_U} \eta$$ for
$U=U(X)$ with $(X,X)>0$ and $\Gamma_U$ the stabilizer of $U$ in
$\Gamma_0(\frak{n})$. The subspace $D_U$ has signature $(1,1)$ and
therefore the stabilizer $\Gamma_U$ is either infinitely cyclic or
trivial. In the first case, the cycle $\Gamma_U \backslash D_U$ is a
closed geodesic and Thom's Lemma holds by \cite{KM88}. When the
stabilizer is trivial, the cycle $D_U$ is an infinite geodesic
joining two cusps.

\begin{thm}
Assume $\Gamma_U$ is trivial. Then \begin{equation}
\label{Thom}\int_{D}  \eta \wedge \varphi(X) = {\rm exp}(-\pi {\rm
tr}(X,X)) \int_{D_U} \eta.\end{equation}
\end{thm}

\begin{proof}
We essentially follow the proof for signature $(2,1)$ in Theorem 7.6
of \cite{FM02} but have to deal with the more complicated Fourier expansion for modular forms over imaginary quadratic fields.
We can assume $$X=(2au_0 +be_2,2cu_0+de_3)$$ with
$d=d'/\sqrt{D}$ and $a,c \in \Q$, $b,d' \in \Q_{>0}$ so that $D_U$
is the geodesic joining the cusps $\infty$ and $\frac{a}{b} +
\frac{c}{d'} \sqrt{-D} \in F$. The stabilizer of the cusp $\infty$
is $\Gamma_{\infty}=\left \{\begin{pmatrix}1&\alpha \\&1
\end{pmatrix} : \alpha \in \mathcal{O} \right\}.$
We have
$$\int_D \eta \wedge \varphi(X)  = \int_{\Gamma_{\infty} \backslash D}
\eta \wedge \sum_{\alpha \in \mathcal{O}} \varphi(X+(2b
\alpha_1 u_0,2d \alpha_2 u_0), (z,r)),$$ where we write $\alpha_1={\rm Re}
(\alpha)$ and $\alpha_2={\rm Im}(\alpha)$. We get an explicit
formula for $\varphi(\begin{pmatrix}1&\alpha \\&1
\end{pmatrix}.X, (z,r))$ by calculating
\begin{eqnarray*}\lefteqn{\varphi\left(\begin{pmatrix}r^{-1/2}&-zr^{-1/2} \\&r^{1/2}
\end{pmatrix}.(X+(2b\alpha_1 u_0,2d \alpha_2 u_0)), (0,1)\right)=}\\&=& \frac{1}{2} e^{-\frac{2 \pi}{r^2}
(a+b(\alpha_1-x))^2} e^{-\pi(b^2+d^2)} e^{-\frac{2 \pi}{r^2}
(c+d(\alpha_2-y))^2} \cdot (\frac{1}{r}(a+b(\alpha_1-x))
\frac{dr}{r}+b \frac{dx}{r}) \wedge(\frac{1}{r}(c+d(\alpha_2-y))
\frac{dr}{r}+d \frac{dy}{r})\\&=:&\left( \varphi_1 dy \wedge dr +
\varphi_2 dx\wedge dy + \varphi_3 dr\wedge dx \right) e^{- \pi
(b^2+d^2)}.
\end{eqnarray*}

The Fourier transform with respect to $\alpha$, which we define as
$$\hat{\varphi}(\beta_1+i \beta_2)= \int_{-\infty}^{\infty}
\int_{-\infty}^{\infty} \varphi(\begin{pmatrix}1&\alpha_1 + i
\alpha_2\\&1
\end{pmatrix}.X, (z,r)) e^{2\pi
i \alpha_1 \beta_1} e^{2 \pi i \alpha_2 \beta_2} d\alpha_1
d\alpha_2$$ is then given by $$\hat{\varphi_1}(\beta,X) dy \wedge dr
+ \hat{\varphi_2}(\beta,X) dx\wedge dy + \hat{\varphi_3}(\beta,X) dr
\wedge dx$$ with
\begin{eqnarray*}\hat{\varphi_2}(\beta,X)&=& e^{-\frac{\pi}{2}
\left(\frac{\beta_1 r}{b}\right)^2} e^{-\frac{\pi}{2}
\left(\frac{\beta_2 r}{d}\right)^2} e^{2 \pi i \beta_1x} e^{2 \pi i
\beta_2y} e^{-2 \pi i \beta_1 \frac{a}{b}} e^{-2
\pi i \beta_2 \frac{c}{d}},\\
\hat{\varphi_1}(\beta,X)&=& -\frac{i \beta_1 r}{2 b^2} \cdot
\hat{\varphi_2}(\beta,X), \\\hat{\varphi_3}(\beta,X)&=& -\frac{i
\beta_2 r}{2 d^2} \cdot \hat{\varphi_2}(\beta,X).\end{eqnarray*}

By the Eichler-Shimura-Harder isomorphism $[\eta]$ can be
represented by a $\Gamma_0(\frak{n})$-invariant harmonic
differential on $\mathbf{H}_3$ of the form $-f_0 \frac{dz}{r}+f_1
\frac{dr}{r}+f_2 \frac{d\overline z}{r}$, where
$$f=(f_0,f_1,f_2): \mathbf{H}_3 \to \C^3$$ is
a weight 2 cusp form for $\Gamma_0(\frak{n})$ (see \cite{U95}
Th\'{e}or\`{e}me 3.2, \cite{HaGL2} \S 3.6 and \cite{CW94}). By
\cite{CW94} (with a correction from \cite{By99} Proposition 100)
such a function $f$ has a Fourier expansion (about the cusp
$(0,\infty)$) of the form
$$f(z,r)=\sum_{\xi \in \vartheta^{-1}} c(\xi) r^2 \mathbf{K}(4
\pi |\xi| r) \cdot {\rm diag}(\xi/|\xi|,1,\overline \xi/|\xi|)
\Psi(\xi z),$$ where $\vartheta$ is the different of $F$,
$\mathbf{K}(t)$ is the vector-valued K-Bessel function
$$\mathbf{K}(t)=(-\frac{1}{2}i K_1(t), K_0(t), \frac{1}{2}i
K_1(t)),$$ and $\Psi(z)={\rm exp}(-2 \pi i (z+ \overline z)).$

By Poisson summation we therefore obtain
$$I:=\int_D \eta \wedge \varphi(X) =\int_{\Gamma_{\infty} \backslash D}
{\rm vol}(\C/\mathcal{O})^{-1} \sum_{\beta \in \vartheta^{-1}}
((-f_0+f_2) \hat{\varphi_1}(\beta,X) + f_1 \hat{\varphi_2}(\beta,X)
- i (f_0+f_2) \hat{\varphi_3}(\beta,X)) dx\wedge dy \wedge dr.$$ We
pick a fundamental domain for $\Gamma_{\infty} \backslash D$ and
integrate with respect to $z=x+iy$, which singles out the Fourier
coefficients:
$$I=e^{- \pi
(b^2+d^2)} \sum_{\xi \in \vartheta^{-1}} c(\xi)
\Psi\left(\xi(\frac{a}{b} + i \frac{c}{d})\right) \int_0^{\infty} r
e^{-2 \pi((\frac{\xi_1 r}{b})^2+(\frac{\xi_2 r}{d})^2)} \left(
K_0(4\pi|\xi|r) + \left(\frac{\xi_1^2}{b^2}+
\frac{\xi_2^2}{d^2}\right) \frac{r}{|\xi|} K_1(4 \pi |\xi|r)\right)
dr.$$ We now use integration by parts for $$\int_0^{\infty} r e^{-2
\pi((\frac{\xi_1 r}{b})^2+(\frac{\xi_2 r}{d})^2)}
\left(\frac{\xi_1^2}{b^2}+ \frac{\xi_2^2}{d^2}\right)
\frac{r}{|\xi|} K_1(4 \pi |\xi|r) dr
$$ and the following two properties of the K-Bessel function (see
\cite{MOS} 3.1.1 and 3.2):$$\frac{d}{dx}(x K_1(x))=-x K_0(x)$$ and
$${\rm lim}_{x \to 0^+} x K_1(x)=1.$$ We conclude that both sides of (\ref{Thom}) equal
$$\frac{1}{(4\pi)^2}{\rm exp}(-\pi {\rm tr}(X,X)) \sum_{\xi \in \vartheta^{-1}} c(\xi)
\Psi\left(\xi (\frac{a}{b} + i \frac{c}{d})\right) {\rm
Nm}(\xi)^{-1}.$$
\end{proof}

\section{Arithmetic properties}

\subsection{Arithmetic Siegel modular forms}

For $\tau=u+iv \in \mathbf{H}_n$ and $g_f \in G(\A_f)$ we define
$$\theta_{\varphi}(\eta)(\tau,g_f)= {\rm det}(v)^{-m/4} \cdot \theta_{\varphi}(\eta)(g_{\tau} g_f),$$
where $g_{\tau}=\begin{pmatrix} v^{1/2} & u v^{-1/2} \\ 0& v^{-1/2}\end{pmatrix} \in G_1(\Real)$.

We recall from \cite{Ha84} (2.2.2.1) that any such holomorphic automorphic form $\phi$ on $\mathbf{H}_n \times (G(\A_f)/L)$ has a
Fourier-Jacobi expansion of the form
$$\phi(\tau,g_f)=\sum_{\beta \in {\rm Sym}_n(\Q)} a_{\beta}(\tau,g_f),$$ where
$$a_{\beta}(\tau,g_f)= \int_{U(\Q) \backslash U(\A)} \phi(u(\tau,g_f)) {\rm exp}(-2 \pi i {\rm tr}(\beta u)) du.$$
One checks (see e.g. \cite{Sug}(1-19)) that
$$a_{\beta}(\tau,g_f)=a_{\beta}(g_f) \cdot {\rm exp}(2 \pi i {\rm
tr}(\beta \tau))$$ for some $a_{\beta}(g_f) \in \C$.

Let $p$ be any rational prime. Fix embeddings $\overline \Q \hookrightarrow \overline \Q_p
\hookrightarrow \C$. Let $M \subset \overline \Q$ be a number field
and $\wp$ the prime of $M$ above $p$. By the $q$-expansion principle
(see \cite{CF}) arithmetic holomorphic automorphic forms are
characterized by their Fourier-Jacobi expansion. We therefore call
$f$ \textit{algebraic} (resp. $M$-rational, resp. $\wp$-integral) if
$a_{\beta}(g)$ lies in $\overline \Q$ (resp. $M$, resp.
$\mathcal{O}_{M_{\wp}}$) for all $\beta \in {\rm Sym}_n(\Q)$ and all
$g\in G(\A_f)$. (In fact, one only needs to consider finitely many
$g \in G(\A_f)$, see \cite{Ha84} \S3 or \cite{T91} \S3.)

Since $R(\A_f)$ preserves $S(V(\A_f)^n)_{\overline \Q}$, Proposition  \ref{SMF} and Theorem \ref{Fourier} imply:
\begin{cor}
If $[\eta] \in H^{(s-n)t}_c(S_K,\overline \Q)$
and $\varphi \in S(V(\A_f)^n)_{\overline \Q}$ a $K_1$-invariant
Schwartz function then $\theta_{\varphi}(\eta)(\tau,g_f)$ is an
algebraic holomorphic Siegel modular form of weight $m/2$.
\end{cor}

\begin{rem}
We can, in fact, replace $\overline \Q$ by $\Q^{\rm ab}$ in the
above statement since the Weil representation is defined over $\Q^{\rm ab}$. In the next section we will give an example of a suitable choice of $\varphi$ taking values in
$\mathcal{O}_{M_{\wp}}$ for which we can prove the $\wp$-integrality of
$\theta_{\varphi}(\eta)$ for $\eta \in {\rm
im}(H^{(s-n)t}_c(S_K,\mathcal{O}_{M_{\wp}}) \to
H^{(s-n)t}_c(S_K,M_{\wp}))$.
\end{rem}

\subsection{Definition of Schwartz function} \label{schwartz}
Recall from Section \ref{orthdefn} the quadratic character $\chi_V$ associated to $V$. Let $X$ be an integral lattice on $V$ and put $X_{\ell}=X
\otimes_{\Z} \Z_{\ell}$ for every prime $\ell$ of $\Z$. We define $\varphi=\prod_{\ell} \varphi_{\ell} \in S(V(\A_f)^n)_{\Z}$ by
putting $\varphi_{\ell}$ equal to the characteristic function of $X_{\ell}^n$. We also fix $K=\prod_{\ell} K_{\ell} \subset H(\A_f)$ with
$$K_{\ell}=\{h \in H(\Q_{\ell}) | h X_{\ell}=X_{\ell} \}.$$
Let $\check{X}_{\ell}=\{x \in V \otimes \Q_{\ell} | (x,y) \in \Z_{\ell}  \, \forall y \in X_{\ell} \}$ and let $({\ell}^{-n_{\ell}})$ be the $\Z_{\ell}$-module
generated by $\{(x,x) | x \in \check{X}_{\ell} \}.$ Set $N=\prod_{\ell} {\ell}^{n_{\ell}}$ (the ``\textit{level of the lattice} $X$") and define $$L(N)=\left\{ \begin{pmatrix}
  A &B\\C&D\end{pmatrix} \in G(\hat{\Z}) | C \equiv 0 \mod{N \hat{\Z}} \right\} \subset G(\A_f).$$

It is easy to see that $\lambda(K) \subset \hat{\Z}^*$. We recall from Eichler the following lemma:

\begin{lem}[\cite{Eich} Satz 11.2]
Assume that $X_{\ell}$ is a $\mathfrak{A}$-maximal lattice for some fractional ideal $\mathfrak{A}$ of $\Z_{\ell}$. If $h \in H(\Q_{\ell})$ satisfies $\lambda(h) \in \Z_{\ell}^*$ then there exists $k \in K_{\ell}$ such that $\lambda(k)=\lambda(h)$.
\end{lem}

\begin{prop}
  Let $[\eta] \in{\rm im}(H^{(s-n)t}_c(S_K,\mathcal{O}_{M_{\wp}}) \to H^{(s-n)t}_c(S_K,M_{\wp}))$. Assume that $X_{\ell}$ is an $\ell^{k_{\ell}}$-maximal lattice for all $\ell$ for some $k_{\ell}$ and $\lambda(H(\Q_p))\supset\Z_p^*$.
  Then $\theta_{\varphi}(\eta)$ is a $\wp$-integral holomorphic Siegel modular form of weight $m/2$ and central character $\chi_V^n$, and (almost) level $N$ and character $\chi_V$ in the sense that \begin{equation} \label{trafo} \theta_{\varphi}(\eta)(gk)=\chi_V({\rm det} A) \theta_{\varphi}(\eta)(g) \text{ for } k=\begin{pmatrix}   A&B\\C&D  \end{pmatrix} \in L(N) \cap G(\A_f)^+.\end{equation}
\end{prop}

\begin{proof}
By the definition of $\theta_{\varphi}(\eta)$ it suffices to check (\ref{trafo}) for $g \in G(\A)^+$.
  Note that under our assumption on the lattice $X$ and the preceding lemma we then have $\theta_{\varphi}(\eta)(gk)=\theta_{\varphi}(\eta)(gk_1)$ by Definition \ref{thetaDefn}. Therefore the statement about the level and character follows from the following Lemma:
  \begin{lem}[{\cite[Lemma~2.1]{Y2}}] \label{Yoslemma}
    $\omega(k_1) \varphi= \chi_V({\rm det} A) \varphi$ for $k_1=\begin{pmatrix} A &B\\C&D\end{pmatrix} \in L(N) \cap G_1(\A_f)$.
  \end{lem}
  By strong approximation we have $G(\A)=G(\Q) L(N) G(\Real)_+$ and so it suffices to check $\wp$-integrality on $L(N)$,
  and by definition of $\theta$ on $L(N) \cap G(\Q) G(\A)^+$.
  If $k \in L(N) \cap G(\A_f)^+$ this follows directly from Theorem \ref{Fourier} and Lemma \ref{Yoslemma}. 
  In general, write $k=\begin{pmatrix} 1&0\\0& \lambda\end{pmatrix} g z_{\infty}^{-1} z_{\infty}$ for $g \in G(A)^+$, $\lambda \in \Q_{>0}$ with $|\lambda|_p=1$ and $z_{\infty}=\lambda^{-1/2} I_m \in G(\Real)^+$. Since the central character has finite order we then have $\theta_{\varphi}(\eta)(k)=\theta_{\varphi}(\eta)(g z_{\infty}^{-1})$ and by Theorem \ref{Fourier} this equals
  $$\sum_{\beta > 0} W_{\beta}(g_{\infty}z_{\infty}^{-1})\cdot \int_{Z(\beta,\varphi',K^{h}_1)} \eta(h)$$ for $h\in H(A_f)$ with $\lambda(h)=\lambda(g z_{\infty}^{-1})$ and $\varphi'=\omega(g_f, h) \varphi=L(h) \omega(\begin{pmatrix} 1&0\\0& \lambda(k)^{-1}\end{pmatrix} k) \varphi$, the latter equality using Lemma \ref{commutation}. Applying Lemma \ref{Yoslemma} for $\begin{pmatrix} 1&0\\0& \lambda(k)^{-1}\end{pmatrix} k$ and noting $|\lambda(h)|_p=1$ we deduce the $\wp$-integrality of $\theta_{\varphi}(\eta)$.
\end{proof}

\begin{rem}
  Assume that for $\ell \nmid N$ the form $\eta$ is an eigenfunction for the Hecke algebra $\mathcal{H}(H_1(\Q_{\ell})//H_1(\Z_{\ell}))$ and that the Witt index of $V \otimes \Q_{\ell}$ (the dimension of the maximal isotropic subspace) is less than or equal to $n$. Rallis' generalisation of the Eichler commutation relation (\cite{Ra82} \S 4.B) implies that $\theta_{\varphi}(\eta)|_{G_1(\A)}$ is an eigenfunction for the Hecke algebra $\mathcal{H}(G_1(\Q_{\ell})//G_1(\Z_{\ell}))$.
\end{rem}

\subsection{Orthogonal Spaces of dimension 4} \label{Hecke}
In this section we restrict to quadratic spaces $V$ with dimension $m=4$ and signature $(3,1)$ and analyze the Hecke properties of the theta lift.

We refer the reader to Section 2 of \cite{Ro01} for a summary of results of four dimensional quadratic spaces.
We consider the following examples:  Let $F$ be an imaginary quadratic field with ring of integers $\mathcal{O}$
and for every place $v$ of $F$ write $\frak{q}_v \subset \mathcal{O}$ for the corresponding prime ideal.
For $D_0$ a quaternion algebra over $\Q$ put $D=D_0 \otimes_{\Q} F$, write $*$ for the main involution of $D$ and denote the natural extension of the non-trivial automorphism of $F$ over $\Q$ to the semi-automorphism of $D$ by $-$. Put $$V=\{x \in D | \overline x=x^* \}$$ with quadratic form $N(x)=x x^*$. This is a four dimensional quadratic space of signature $(3,1)$ since $D$ necessarily splits at $\infty$.

Let $N_1= \prod_{D_v \text{ramified}} \mathfrak{q}_v $, which is a product of split primes of $\Z$.
Let $R \subset D$ be an Eichler order of level $N=N_1 N_2$ for $N_2 \in \Z$ squarefree and coprime to $N_1$, i.e. $R_v$ is a maximal order in $D_v$
 for all $v \nmid N_2$ and is conjugate to $$\left\{\begin{pmatrix}
   a&b\\c&d \end{pmatrix} \in M_2(\mathcal{O}_v) | c \equiv 0 \mod{\mathfrak{q}_v} \mathcal{O}_v \right\}$$ for $v\mid N_2$ (where $D_v$ has been identified with $M_2(F_v)$).

We define a lattice $X:=\{x \in R | \overline x=x^*\} \subset V$, with corresponding $\varphi$ and $K$ as in Section \ref{schwartz}.
Note that the level of $X$ (as defined in Section \ref{schwartz}) equals $N$. The lattices $X_{\ell}$ are all $\ell$-maximal.

From \cite{Ro01} and \cite{Ro94} we deduce that $L(N) \cap G(\A_f)^+$ differs from $L(N)$ only at the places $\ell$ ramified in $F/\Q$. For $\ell$ split in $F/\Q$ one has $\lambda(H(\Q_{\ell}))=\Q_{\ell}^*$, so $G(\Q_{\ell})=G(\Q_{\ell})^+$. From the explicit description of $V \otimes \Q_{\ell}$ given on \cite{Ro01} p.273 for $\ell \mid {\rm disc}(D)$ one deduces the Witt decomposition $V \otimes \Q_{\ell} \cong (F \otimes \Q_{\ell}) \bot \mathbf{H}$ (see also \cite{Ro94} Section 3). This implies that $\lambda(H(\Q_{\ell}))={\rm Nm}^{\Q_{\ell}\otimes F}_{\Q_{\ell}}((\Q_{\ell} \otimes F)^*)$, so for $\ell$ inert in $F/\Q$ we have $G(\Z_{\ell}) \subset G(\Q_{\ell})^+$, whereas $[G(\Z_{\ell}):G(\Q_{\ell})^+ \cap G(\Z_{\ell})]=2$ for $\ell$ ramified in $F/\Q$.

\begin{cor}
  Assume $p \nmid {\rm disc}(F/\Q)$. Let $[\eta] \in{\rm im}(H^{(s-n)t}_c(S_K,\mathcal{O}_{M_{\wp}}) \to H^{(s-n)t}_c(S_K,M_{\wp}))$.
  Then $\theta_{\varphi}(\eta)$ is a $\wp$-integral holomorphic Siegel modular form of weight $2$ and trivial central character, and (almost) level $N$ and character $\chi_V$ in the sense that \begin{equation}  \theta_{\varphi}(\eta)(gk)=\chi_V({\rm det} A) \theta_{\varphi}(\eta)(g) \text{ for } k=\begin{pmatrix}   A&B\\C&D  \end{pmatrix} \in L(N) \text{ with } \lambda(k) \in {\rm Nm}^{\A_F^*}_{\A^*}(A_F^*).\end{equation}
\end{cor}

Note that $D^*$ defines an algebraic group over $F$. Denote its restriction of scalars from $F$ to $\Q$ by $\tilde D^*$. It acts on $V$ via $x \mapsto g x \overline g^*$. We refer the reader to \cite{HST} \S 1 for a description of the relationship between cuspidal automorphic forms of ${\rm GO}(V)_{\A}=H_{\A}$ and $\tilde D^*_{\A}$ for $D=M_2(F)$. One can generalize this to all our quaternion algebras $D$ using the characterisation of the special similitude group ${\rm GSO}(V)$ given in \cite {Ro01} Theorem 2.3 and Proposition 2.7.
We are going to use that an automorphic form for $H_{\A}$ arises from an automorphic form $f_{\eta} \in \mathcal{A}(\tilde D^*_{\A})$ (not uniquely, see \cite{HST} Proposition 2).

For every place $v$ of $F$ such that $v \nmid N$ we define Hecke operators $T'(v)$ acting on $f_{\eta} \in \mathcal{A}(\tilde D^*_{\A})$ as follows: We may assume that $R_v=R \otimes_{\mathcal{O}} \mathcal{O}_v$ is mapped onto $M_2(\mathcal{O})$ when we fix a splitting $D \otimes F_v \cong M_2(F_v)$. Let $\pi_v$ be a prime element of $F_v$ and let $R_v^* \begin{pmatrix}
  1&0\\0& \pi_v\end{pmatrix} R_v^*= \bigcup_i h_i R_v^*.$ Then define $$(T'(v)f_{\eta})(h)=\sum_i f_{\eta}(h h_i) \text{ for } h \in \tilde D^*_{\A}.$$

Now consider the Hecke action on Siegel modular forms. Let $L(N)=\prod_{\ell} L(N)_{\ell}$ for $L(N)_{\ell} \subset G(\Z_{\ell})$ with $L(N)_{\ell} =G(\Z_{\ell})$ for ${\ell} \nmid N$. For any $\ell$ the double coset $L(N)_{\ell} M L(N)_{\ell}= \bigcup_i g_i L(N)_{\ell}$ for $M \in G(\Q_{\ell})$ acts on $\phi:G(\A) \to \C$ by $(L(N)_{\ell} M L(N)_{\ell}) \phi(g)=\sum_i \phi(g g_i)$.
We single out the operators $$T_{\ell}=L(N)_{\ell} {\rm diag}(\underbrace{{\ell}, \ldots, {\ell}}_{n},\underbrace{1,\ldots, 1}_n) L(N)_{\ell}$$ and $$R_{\ell}^{(s)}=L(N)_{\ell} {\rm diag}(\underbrace{{\ell},\ldots, \ell}_{n-s},\underbrace{1,\ldots, 1}_{s},\underbrace{{\ell},\ldots, \ell}_{n-s}, \underbrace{{\ell}^2, \ldots, \ell^2}_s) L(N)_{\ell} \text{ for } 0 \leq s \leq n.$$

\begin{thm}
   Let $[\eta] \in H^{(s-n)t}_c(S_K,\C) $ correspond to  $f_{\eta} \in \mathcal{A}(\tilde D^*_{\A})$ and
   let $\varphi=\prod_{\ell} \varphi_{\ell} \in S(V(\A_f)^n)_{\Z}$ with $\varphi_{\ell}$ the characteristic function of $X_{\ell}$.

   Consider $\ell \nmid N_2$ unramified in $F/\Q$.  If ${\ell} \nmid N_1$  then assume that $f_{\eta}$ is an eigenform for the Hecke operator $T'(v)$ with eigenvalue $\lambda_v$ for all $v \mid {\ell}$.  If ${\ell} \mid N_1$ then assume that $f_{\eta}$ is an eigenform with eigenvalue $\pm 1$ for the Atkin-Lehner involution given by right multiplication by $(\pi,1)$  for $\pi$ a prime element of $D_{0,{\ell}}$.

   Then $\theta_{\varphi}(\eta)$ is an eigenfunction for the Hecke operators $T_{\ell}$ and $R_{\ell}=R_{\ell}^{(1)}$ with eigenvalues for $n=2$ given by:
$$\begin{tabular}{c|c|c}
  & $T_{\ell}$ & $R_{\ell}$\\ \hline \hline
  $({\ell})=\frak{{l}} \overline{\frak{{l}}}$ split, ${\ell} \nmid N$ & ${\ell}(\lambda_{\frak{{l}}}+\lambda_{\overline{\frak{{l}}}})$ & $({\ell}^2-1)+{\ell} \lambda_{\frak{{l}}} \lambda_{\overline{\frak{{l}}}}$ \\ \hline
  ${\ell}$ inert, ${\ell} \nmid N$ & $0$ & $({\ell}^2+1)+{\ell} \lambda_{\ell}$\\ \hline
  ${\ell} \mid N_1$ & $\pm {\ell}$ & ${\ell} (\ell+1)$\\
\end{tabular}$$
\end{thm}

\begin{proof}
For $\ell \nmid N$ and $n=2$ this can be deduced from   \cite{HST} Lemmata 10,11, where we set the auxiliary $\delta_p=+1$  to ensure local non-vanishing. For similar calculations see \cite{Y1} Theorem 5.2, \cite{BSP} Theorem 6.1, and \cite{U98} Th\'{e}or\`{e}me 3.3.5.

For $\ell \mid N_1$ we note that $\ell$ splits in $F/\Q$ so that $D_{\ell} \cong D_{0,\ell} \times D_{0,\ell}$ and $X_{\ell} \cong R_{0,\ell}$. We can therefore refer to the (local) calculation of the $R_{\ell}$ eigenvalue in \cite{BSP} Lemma 7.3 (b), after multiplying by ${\rm diag}(\ell^{-1}, \ldots, \ell^{-1}) \in Z(G(\Q))$ and observing that $\chi_{V,\ell}=1$.

We calculate the $T_{\ell}$ action by first observing that $$L(N)_{\ell} {\rm diag}(\underbrace{{\ell}, \ldots, {\ell}}_{n},\underbrace{1,\ldots, 1}_n) L(N)_{\ell}=\bigcup_B \begin{pmatrix}
  \ell I_n&B\\0&I_n\end{pmatrix}L(N)_{\ell},$$ where $B$ runs over the $\ell^{n(n+1)/2}$ representatives for the symmetric matrices modulo $\ell$.
  Let $h \in H(\A)$ with $\lambda(h)=\lambda(g)$ and $h' \in G(\Q_{\ell})$ with $\lambda(h')=\ell$. Then by definition (it suffices to check this for $g \in G(\A)^+$)
  $$\theta_{\varphi}(\eta)(g \begin{pmatrix}  \ell I_n &B\\0& I_n\end{pmatrix})= \int_{S^1_{K^{h h'}}} \eta(hh') \wedge \theta((g \begin{pmatrix}
  \ell I_n & B\\0& I_n\end{pmatrix})_1, L(h h') \tilde \varphi).$$ By \cite{HK92} Lemma 5.1.7 (a) we have $$\theta((g \begin{pmatrix}
  \ell I_n &B\\0& I_n\end{pmatrix})_1, L(h h') \tilde \varphi)= \theta(g_1,L(h) \omega((\begin{pmatrix}
  \ell I_n &B\\0& I_n\end{pmatrix})_1) L(h') \tilde \varphi).$$
By Lemma \ref{commutation} we know that $\omega((\begin{pmatrix}
  \ell I_n &B\\0& I_n\end{pmatrix})_1) L(h') =L(h') \omega(\begin{pmatrix}  \ell I_n &B\\0& \ell^{-1} I_n\end{pmatrix})$. As in the proof of \cite{BSP} Lemma 7.3 (a) we calculate $$\omega(\begin{pmatrix}  \ell I_n &B\\0& \ell^{-1} I_n\end{pmatrix}) \varphi(x)= \frac{1}{\ell^{2n}} \psi_{\ell}({\rm tr}(\ell B (x,x))) \varphi(x \begin{pmatrix}  \ell  &0\\0& \ell\end{pmatrix})= \frac{1}{\ell^{2n}} \varphi( \pi x).$$
  We identify $V \otimes \Q_{\ell} \cong D_{0,\ell}$ and note that $GSO(D_{0,\ell}) \cong (D_{0,\ell} \times D_{0,\ell})/\Q_{\ell}$, with $(h_1,h_2)$ acting by $x \mapsto h_1 x h_2^{-1}$ (see \cite{BSP} p. 60). We now assume that $h'=(\pi,1)$, by which we denote the element of $H(\Q_{\ell})$ that acts on $x \in V \otimes \Q_{\ell}$ by $x \mapsto \pi  x$, and so $L(h') \varphi(\pi x)= \ell^2 \varphi(x)$.
  Since $\pi R_{0, \ell} \pi^{-1} =R_{0,\ell}$ we have $K^{h h'}=K^h$ and hence $$\theta_{\varphi}(\eta)(g \begin{pmatrix}  \ell I_n &B\\0& I_n\end{pmatrix})= \frac{1}{\ell^{2(n-1)}} \theta_{\varphi}(\eta((\pi,1)))(g).$$
\end{proof}

\bibliographystyle{amsalpha}
\bibliography{theta}

\end{document}